%% file: paper_ArXiv.tex
\documentclass{siamart171218}
\usepackage{amsfonts, amsmath, amssymb}
\usepackage{bm}
\usepackage{graphicx}
\usepackage{epstopdf}
\usepackage{algorithm}
\usepackage{algpseudocode}
\usepackage{pgfplots}
\usepackage{pgfplotstable}
\usepackage[ddmmyy,hhmmss]{datetime}
\usepackage{enumitem}
\ifpdf
\DeclareGraphicsExtensions{.eps,.pdf,.png,.jpg}
\else
\DeclareGraphicsExtensions{.eps}
\fi

\newsiamremark{remark}{Remark}
\newsiamremark{hypothesis}{Hypothesis}
\crefname{hypothesis}{Hypothesis}{Hypotheses}
\newsiamthm{claim}{Claim}
\headers{Compress-and-restart for matrix equations}{}

\title{Compress-and-restart block Krylov subspace methods for Sylvester matrix equations} 

\author{Daniel Kressner\thanks{EPF Lausanne, Switzerland, \email{daniel.kressner@epfl.ch}}
	\and Kathryn Lund\thanks{Charles University, Prague, Czech Republic, \email{kathryn.lund@karlin.mff.cuni.cz}.  Supported in part by the Charles University PRIMUS grant, project no. PRIMUS/19/SCI/11, and in part by the SNSF research project \emph{Fast algorithms from low-rank updates}, grant number: 200020\_178806.}
	\and Stefano Massei\thanks{EPF Lausanne, Switzerland,
		\email{stefano.massei@epfl.ch}. Supported by the SNSF research project \emph{Fast algorithms from low-rank updates}, grant number: 200020\_178806.}
	\and
	Davide Palitta\thanks{Research Group Computational Methods in Systems
and Control Theory (CSC),
Max Planck Institute for Dynamics of Complex Technical Systems,
Sandtorstra\ss{e} 1, 39106 Magdeburg, Germany,
		\email{palitta@mpi-magdeburg.mpg.de}. } 
}

\usepackage{amsopn}

\input{macros}

\begin{document}

\maketitle

\begin{abstract}
Block Krylov subspace methods (KSMs) comprise building blocks in many state-of-the-art solvers for large-scale matrix equations as they arise, e.g., from the discretization of partial differential equations. While extended and rational block Krylov subspace methods provide a major reduction in iteration counts over polynomial block KSMs, they also require reliable solvers for the coefficient matrices, and these solvers are often iterative methods themselves. It is not hard to devise scenarios in which the available memory, and consequently the dimension of the Krylov subspace, is limited.  In such scenarios for linear systems and eigenvalue problems, restarting is a well explored technique for mitigating memory constraints.  In this work, such restarting techniques are applied to polynomial KSMs for matrix equations with a compression step to control the growing rank of the residual.  An error analysis is also performed, leading to heuristics for dynamically adjusting the basis size in each restart cycle.  A panel of numerical experiments demonstrates the effectiveness of the new method with respect to extended block KSMs.
\end{abstract}




\input{intro}
\input{algorithm}
\input{error_analysis}
\input{numerical_ex_arXiv}
\input{conclusions}


\bibliographystyle{siamplain}
\bibliography{restarts_mat_eq}
\end{document}

%% file: macros.tex
\newcommand{\tolres}{\varepsilon}
\newcommand{\tolcomp}{\delta}
\newcommand{\memmax}{\texttt{mem}_{\max}}
\newcommand{\memmaxmath}{\mathtt{mem}_{\max}}
\newcommand{\flagconv}{\texttt{flag}_{\text{conv}}}

\newcommand{\ve}{\bm e}

\newcommand{\vv}{\bm v}

\newcommand{\vC}{\bm C}
\newcommand{\vCtil}{\widetilde{\vC}}
\newcommand{\vD}{\bm D}
\newcommand{\vDtil}{\widetilde{\vD}}
\newcommand{\vE}{\bm E}

\newcommand{\vU}{\bm U}
\newcommand{\vV}{\bm V}

\newcommand{\vX}{\bm X}
\newcommand{\vY}{\bm Y}
\newcommand{\vZ}{\bm Z}

\newcommand{\bUU}{\bm{\mathcal{U}}}
\newcommand{\bVV}{\bm{\mathcal{V}}}

\newcommand{\GG}{\mathcal{G}}

\newcommand{\HH}{\mathcal{H}}

\newcommand{\KK}{\mathcal{K}}

\newcommand{\Rerr}{\Delta R}
\newcommand{\Sigmatil}{\widetilde{\Sigma}}
\newcommand{\Util}{\widetilde{U}}
\newcommand{\Vtil}{\widetilde{V}}
\newcommand{\Wtil}{\widetilde{W}}
\newcommand{\Xtil}{\widetilde{X}}
\newcommand{\Zerr}{\Delta Z}
\newcommand{\Ytil}{\widetilde{Y}}

\newcommand{\spK}{\mathcal K} 

\newcommand{\C}{\mathbb C}

\newcommand{\R}{\mathbb R}

\DeclareMathOperator{\diag}{diag} %
\DeclareMathOperator{\colspan}{colspan} %

\DeclareMathOperator{\rank}{rank} %

\renewcommand{\Re}{\mathrm{Re}}

\newcommand{\inv}{{-1}}

\newcommand{\zero}{{(0)}}
\newcommand{\one}{{(1)}}
\newcommand{\kbar}{\bar{k}}
\renewcommand{\k}{{(k)}}
\newcommand{\kless}{{(k-1)}}
\newcommand{\kmore}{{(k+1)}}
\renewcommand{\j}{{(j)}}
\newcommand{\half}{{1/2}}

\newcommand{\norm}[1]{\left\lVert#1\right\rVert}

\newcommand{\MATLAB}{\textsc{Matlab}~}

%% file: intro.tex
\section{Introduction}
This work is concerned with numerical methods for solving 
large-scale Sylvester matrix equations of the form
\begin{equation} \label{eq:sylvester}
A X + X B + \vC \vD^* = 0,
\end{equation}
with coefficient matrices $A, B \in \C^{n\times n}$ and $\vC, \vD\in\mathbb C^{n\times s}$. Although it is essential that both $A,B$ are square, it is not essential that they are of the same size. The latter assumption has been made to simplify the exposition; our developments easily extend to square coefficients $A,B$ of different sizes.

Throughout this paper, we will assume $s\ll n$ and view $\vC, \vD$ as block vectors. This not only implies that the right-hand side $\vC \vD^*$ has low rank, but it also implies that, under suitable additional conditions on the coefficients, such as the separability of the spectra of $A$ and $-B$, the desired solution $X$ admits accurate low-rank approximations; see, e.g.,~\cite{BakerEmbreeSabino2015,Grasedyck2004a,Penzl2000}.

The Sylvester equation~\eqref{eq:sylvester} arises in a variety of applications. In particular, model reduction~\cite{Antoulas2005} and robust/optimal control~\cite{Zhou1996} for control problems governed by discretized partial differential equations (PDEs) give rise to Sylvester equations that feature very large and sparse coefficients $A,B$. Also, discretized PDEs themselves can sometimes be cast in the form~\eqref{eq:sylvester} under suitable separability and smoothness assumptions~\cite{Palitta2016}. Another source of applications are linearizations of nonlinear problems, such as the algebraic Riccati equation~\cite{Bini2012}, and dynamic stochastic general equilibrium models~\cite{Binning2013}, where the solution of~\eqref{eq:sylvester} is needed in iterative solvers. We refer to the surveys~\cite{Benner2013,Simoncini2016} for further applications and details.

Several of the applications mentioned above potentially lead to coefficients $A,B$ that are too large to admit any sophisticated operation besides products with (block) vectors. In particular, the sparse factorization of $A,B$ and, consequently, the direct solution of linear systems involving these matrices may be too expensive. This is the case, for example, when sparse factorizations lead to significant fill-in, as for discretized three-dimensional PDEs~\cite{Davis2016}, or when $A,B$ are actually dense, as for discretized boundary integral equations~\cite{Sauter2011}. In these situations, the methods available for solving~\eqref{eq:sylvester} essentially boil down to Krylov subspace methods. One of the most common variants, due to Saad~\cite{Saad1990} and Jaimoukha and Kasenally~\cite{Jaimoukha1994}, first constructs orthonormal bases for the two block Krylov subspaces associated with the pairs $A,\vC$ and $B^*,\vD$, respectively, and then obtains a low-rank approximation to $X$ via a Galerkin condition. In the following, we will refer to this method as the standard Krylov subspace method (SKSM).

The convergence of SKSM can be expected to be slow when the separation between the spectra of $A$ and $-B$ is small. This happens, for example, when $A$ and $B$ are symmetric positive definite and ill-conditioned, which is usually the case for discretized elliptic PDEs. Slow convergence makes SKSM computationally expensive. As the number of iterations increases, the memory requirements as well as the cost for orthogonalization and extracting the approximate solution increase. Restarting is a common way to mitigate these effects when solving linear systems with, e.g., GMRES~\cite{Saad1986}. In principle, the idea of restarting can be directly applied to Sylvester equations because~\eqref{eq:sylvester} can be recast as a linear system of size $n^2\times n^2$. However, as we will discuss in more detail in Section~\ref{sec:algorithm}, such  an approach increases the ranks of the right-hand side and is, in turn, by itself not well suited for SKSM.  In this work, we propose and partly analyze an algorithm that avoids this problem by combining restarting with compression.

The \emph{ADI method}~\cite{ADI} and \emph{rational Krylov subspace methods (RKSM)} \cite{Benner2009} often converge faster in situations where SKSM struggles. However, this improvement comes at the expense of having to solve (shifted) linear systems with $A,B$ in each iteration. As we are bound to the use of matrix-vector products with $A,B$,
iterative methods, such as CG and GMRES, need to be used for solving these linear systems. These inner solvers introduce another error, and the interplay between this inexactness and the convergence of the outer method has been recently analyzed in~\cite{Kuerschner2018}. Two major advantages of the inner-outer paradigm are that it allows for the use of well established linear solvers and the straightforward incorporation of preconditioners. On the other hand, it also has the disadvantage that it is difficult to reuse information from one outer iteration to the next, often leading to a large number of matrix-vector products overall. Moreover, the convergence of the outer scheme strongly depends on the selection of the shift parameters that define the rational functions used within ADI and RKSM. In the case that the chosen shifts do not result in fast convergence, then the disadvantages of SKSM, such as high memory requirements, still persist. Several strategies for choosing the shifts adaptively have been proposed; see~\cite{Benner2014,Druskin2011} and the references therein.  An a priori strategy for selecting shift parameters, which does not take spectral information of $A,B$ into account, is adopted in the so-called \emph{extended Krylov subspace method (EKSM)} \cite{Breiten2016, Simoncini2007}. EKSM is a particular case of RKSM,\footnote{EKSM is, however, often implemented using block Gram-Schmidt, unlike RKSM, which builds one column at a time.} where half of the shifts are set to $0$ and the remaining are set to $\infty$.   Other large-scale methods for Sylvester equations that require the solution of shifted linear systems and could be combined with iterative solvers in an inner-outer fashion include  the approximate power iteration from~\cite{Hodel1996} and the restarted parameterized Arnoldi method in \cite{AhmadJaimoukhaFrangos2010}.

Besides the inner-outer solvers discussed above, only a few attempts have been made to design mechanisms to limit the memory requirements of Krylov subspace methods for Sylvester equations. In~\cite{SorensenAntoulas2002}, a connection to implicitly restarted Arnoldi method for eigenvalue problems is made.
For symmetric (positive definite) $A, B$, a two-pass SKSM, such as the one discussed in \cite{Kressner2008}, could be used. During the first pass, only the projected equation is constructed and solved; in the second pass, the method computes the product of the Krylov subspace bases with the low-rank factors of the projected solution. 

The rest of this paper is structured as follows. In Section~\ref{sec:algorithm}, we describe the compress-and-restart method proposed in this work for Sylvester equations as well as its adaptation to the special case of Lyapunov matrix equations. Section~\ref{sec:error} provides an error analysis for the approximate solution obtained with our method. Finally, in Section~\ref{sec:numexperiments}, the performance of our method is demonstrated and compared to combinations of EKSM with iterative linear solvers.

%% file: algorithm.tex
\section{Restarted SKSM with compression for linear matrix equations} \label{sec:algorithm}
We first recall the general framework of projection methods for Sylvester equations and then explain our new restarted variant.

\subsection{Projection methods and SKSM} \label{sec:sksm}
Projection methods for solving the Sylvester equation~\eqref{eq:sylvester} seek an approximate solution of the form \begin{equation} \label{eq:approximatesoln}
\Xtil = \bUU_m \Ytil \bVV_m^*,
\end{equation}
where the columns of $\bUU_m$ and $\bVV_m$ form orthonormal bases of suitably chosen (low-dimensional) subspaces. The small core factor $\Ytil$ is determined by, e.g., imposing a Galerkin condition on the residual matrix $R = A \Xtil + \Xtil B+ \vC \vD^*$; see, e.g., \cite{Simoncini2016}.

In SKSM, the subspaces determining~\eqref{eq:approximatesoln} are block Krylov subspaces. More specifically,
$\bUU_m = [\vU_1 | \cdots | \vU_m]$, $\bVV_m = [\vV_1 | \cdots | \vV_m]\in\mathbb{R}^{n\times ms}$; $\vU_i, \vV_i \in \mathbb{R}^{n\times s}$; and $i = 1, \ldots, m$, are such that
\[
\text{range}(\bUU_m) = \KK_m(A,\vC) = \colspan\{ \vC, A\vC, A^2\vC,\ldots, A^{m-1} \vC\} \subset \C^n
\]
and, analogously, $\text{range}(\bVV_m) = \KK_m(B^*,\vD)$.\footnote{Note that in the block Krylov subspace literature, often a distinction is made between the ``true" block Krylov subspace with elements in $\C^{n \times s}$ and the span of all the columns of the basis vectors \cite{FrommerLundSzyld2017, Gutknecht2007}.  For the sake of simplifying notation, $\KK_m(A,\vC)$ and $\KK_m(B^*,\vD)$ will always denote the latter here, meaning they are subspaces of $\C^n$.} If the block Arnoldi process is used for obtaining these bases, then the following block Arnoldi relations hold:
\begin{equation} \label{eq:arnoldirelations}
A \bUU_m = \bUU_m \HH_m + \vU_{m+1} H_{m+1,m} \vE_m^*,
\qquad
B^* \bVV_m = \bVV_m \GG_m + \vV_{m+1} G_{m+1,m} \vE_m^*,
\end{equation}
with $ms\times ms$ block Hessenberg matrices $\HH_m$ and $\GG_m$; $s \times s$ matrices $H_{m+1,m}$ and $G_{m+1,m}$;
and $\vE_m^* = [0_s | \cdots | 0_s | I_s] = \ve_m^* \otimes I_s$, where $0_s$ and $I_s$ denote the $s \times s$ zero and identity matrices, respectively. We refer to, e.g., \cite{FrommerLundSzyld2017, Gutknecht2007} for more details on block Krylov subspaces. The matrix $\Ytil$ is obtained by solving the projected equation
\begin{equation} \label{eq:project_symm_lyap_rank1}
\HH_m Y + Y \GG_m^* + \big(\bUU_m^* \vC \big) \big(\bVV_m^* \vD \big)^* = 0.
\end{equation}
This is again a Sylvester equation of the form~\eqref{eq:sylvester} and, since $m$ should be much smaller than $n$, the equation is of moderate size. Therefore a direct solver, such as the Bartels-Stewart method \cite{Bartels1972}, can be employed for its solution.

Throughout the above discussion, we assumed that none of the block basis vectors $\vU_i$ or $\vV_i$, $i = 1, \ldots, m$, is rank-deficient.  Numerical rank-deficiency is rare in practice, but so-called ``inexact" rank-deficiency may benefit from a deflation or column-replacement procedure; see, e.g., \cite{Birk2015, Soodhalter2017}.

Having the solution $\widetilde Y$ of~\eqref{eq:project_symm_lyap_rank1} at hand, we can compute the Frobenius norm of the residual matrix at a low cost as explained in~\cite{residual,Simoncini2016}:
\begin{equation}\label{eq.res_F}
\norm{R}_F^2 = \big\|{H_{m+1,m} \vE_m^* \Ytil}\big\|_F^2 + \big\|{\Ytil \vE_m G_{m+1,m}^*}\big\|_F^2.
\end{equation}
Similarly, for the spectral norm we have
\begin{equation}\label{eq.res_2}
\norm{R}_2 = \max\big\{\big\|{H_{m+1,m} \vE_m^* \Ytil}\big\|_2, \big\|{\Ytil \vE_m G_{m+1,m}^*}\big\|_2\big\}. 
\end{equation}

Once the residual norm is sufficiently small, the approximation $\Xtil$ from \eqref{eq:approximatesoln} is returned. Note that $\Xtil$ is a large, dense matrix, which is always kept in factored form. Specifically, given a factorization $\Ytil = \vY_L \vY_R^*$, which can be computed, e.g., via a truncated singular value decomposition (SVD), we store the low-rank factors $\vX_L = \bUU_m \vY_L$ and $\vX_R = \bVV_m \vY_R$ of $\Xtil = \vX_L \vX^*_R$, where $L$ and $R$ here simply denote ``left" and ``right," respectively.

\subsection{Restarts and compression}
We now present the new restarted procedure. Suppose that $m_0$ iterations of SKSM have been performed, leading to an approximate, possibly quite inaccurate solution $X^\zero = \vX_L^\zero \big(\vX_R^\zero \big)^*$ with
\[
Y^\zero = \vY_L^\zero \big(\vY_R^\zero \big)^*, \quad
\vX_L^\zero = \bUU_{m_0} \vY_L^\zero, \mbox{ and } 
\vX_R^\zero = \bVV_{m_0} \vY_R^\zero.
\]
For a linear system, a correction to a given approximate solution is obtained by solving the linear system (approximately) with the right-hand side replaced by the residual; see, e.g.,~\cite{Higham2002}. Applying this principle to~\eqref{eq:sylvester}, one first solves
\begin{equation} \label{eq:residual_equation_symm_lyap}
A Z^\one + Z^\one B + R^\zero = 0,
\end{equation}
with $R^\zero := A \vX_L^\zero \big(\vX_R^\zero \big)^* + \vX_L^\zero \big(\vX_R^\zero \big)^* B + \vC \vD^*$, and then adds the correction $Z^\one$ to $X^\zero$ in order to obtain $X^\one$.

The Arnoldi relations~\eqref{eq:arnoldirelations} imply that the residual matrix $R^\zero$
admits the following low-rank representation:
\begin{align*}
	R^\zero
	= &\, A \ \bUU_{m_0} Y^\zero \bVV_{m_0}^* + \bUU_{m_0} Y^\zero \bVV_{m_0}^* B + \vC \vD^* \\
	= &\, \bUU_{m_0} (\HH_{m_0} Y^\zero + Y^\zero \GG_{m_0}^*+ (\bUU_{m_0}^* \vC)(\bVV_{m_0}^* \vD)^*) \bVV_{m_0}^* \\ 
	  &\, + \vU_{m_0+1} H_{m_0+1,m_0} \vE_{m_0}^* Y^\zero \bVV_{m_0}^* 
			+ \bUU_{m_0} Y^\zero \vE_{m_0} G_{m_0+1,m_0}^* \vV_{m_0+1}^* \\
	= &\, \big[ \vU_{m_0+1} H_{m_0+1,m_0} \,|\, \bUU_{m_0} Y^\zero \vE_{m_0} G_{m_0+1,m_0}^*\big]
		  \big[\bVV_{m_0} Y^\zero \vE_{m_0} \,|\,  \vV_{m_0+1}\big]^* \\
	:= &\, \vC^\one \big(\vD^\one \big)^*.
\end{align*}
Because $\vC^\one$, $\vD^\one$ have $2s$ columns, this shows that $R^\zero$ has rank at most $2s$ and, in turn, we can again employ a block Krylov subspace method for solving~\eqref{eq:residual_equation_symm_lyap}. Note, however, that the Krylov subspaces are different from the ones used for $X^\zero$. In order to distinguish them more clearly, we will from now on add the superscript $^\zero$ to quantities generated by SKSM for constructing $X^\zero$; that is, we will write $\bUU^\zero_{m_0}, \bVV^\zero_{m_0}$ instead of $\bUU_{m_0}, \bVV_{m_0}$.  SKSM applied to the correction equation~\eqref{eq:residual_equation_symm_lyap} constructs orthonormal bases $\bUU^\one_{m_1}$ and $\bVV^\one_{m_1}$ for $\KK_{m_1}(A,\vC^\one)$ and $\KK_{m_1}(B^*,\vD^\one)$, respectively. The correction is again returned in factorized form as
$Z^\one = \vZ^\one_L \big(\vZ^\one_R \big)^*$, with $Y^\one = \vY^\one_L \big(\vY^\one_R \big)^*$, $\vZ^\one_L = \bUU^\one_{m_1} \vY^\one_L$, $\vZ^\one_R = \bVV^\one_{m_1} \vY^\one_R$.

The procedure can be iterated; i.e., we can perform multiple restarts. After $k$ restarts, the approximate solution is given by 
\begin{equation} \label{eq:calXk}
	X^\k
	:= \sum_{j = 0}^k Z^\j
	= \big[\bUU^\zero_{m_0} \vY_L^\zero \,| \cdots |\, \bUU^\k_{m_k} \vY_L^\k \big]
	\big[\bVV^\zero_{m_0} \vY_R^\zero \,| \cdots |\, \bVV^\k_{m_k} \vY_R^\k \big]^*.
\end{equation}

The residual for $X^\k$ is equal to the one provided by the last term $Z^\k$ in the summation. Indeed,
\begin{align*}
  A X^\k + X^\k B + \vC \vD^*
  & = A \bigg(\sum_{j=0}^k Z^\j \bigg) + \bigg(\sum_{j=0}^k Z^\j \bigg) B + \vC \vD^* \\
  & = A \bigg(\sum_{j=1}^k Z^\j \bigg) + \bigg(\sum_{j=1}^k Z^\j \bigg) B + R^\zero \\
  & = \ldots\\
  & = A Z^\k + Z^\k B + R^\kless
  = R^\k.
\end{align*}
This relationship can be exploited to design efficient convergence checks within the restarted procedure; see also the discussion in Section~\ref{sec:error}. See \cite{Benner2014a} for a similar procedure within the squared Smith method for certain large-scale Smith equations.

An evident shortcoming of the described procedure is that every time we restart, the rank of the residual may double, thus leading to increasingly large Krylov bases that will inevitably exceed memory capacity. This issue can be mitigated by compressing the residual factors before constructing the Krylov subspace in each cycle. For this task, a well-known QR-SVD-based technique can be employed; see, e.g., \cite[Section 2.2.1]{Kressner2011}. We report such a scheme in Algorithm~\ref{alg:compression} for completeness. Note that there is some flexibility in the choice of truncation criterion in line~\ref{line:truncation}.

In this work, we consider two possibilities. Truncating all singular values below a chosen tolerance $\delta > 0$ implies that the spectral norm truncation error is bounded by $\delta$. Alternatively, the Frobenius norm of the error is bounded by $\delta$ if we make sure that the Euclidean norm of the truncated singular value remains below $\delta$.

The described compression is not only applied to the residual but also each time when the approximate solution~\eqref{eq:calXk} is updated. The impact of these compressions on the quality of the computed solution is discussed in Section~\ref{sec:error}.
\begin{algorithm}
	\caption{Compression of $\vC \vD^*$\label{alg:compression}}
	\begin{algorithmic}[1]
		\Procedure{Compress}{$\vC$, $\vD$, $\tolcomp$}
		\State Compute economy-size QR decomposition $\vC = Q_C R_C$
		\State Compute economy-size QR decomposition $\vD = Q_D R_D$
		\State Compute SVD $R_C R_D^* = U \Sigma V^*$
		\State Truncate $U \Sigma V^* \approx \Util \Sigmatil \Vtil^*$ up to $\tolcomp$ \label{line:truncation}
		\State \Return $\vCtil:= Q_C \Util \Sigmatil^\half$ and $\vDtil:= Q_D \Vtil \Sigmatil^\half$
		\EndProcedure
	\end{algorithmic}
\end{algorithm}

Another measure to make sure that memory consumption stays moderate is to impose a maximum number of iterations $m_k$ in each cycle of SKSM. As the memory requirements are primarily dictated by the number of basis vectors that need to be stored in $\bUU^\k_{m_k}$ and $\bVV^\k_{m_k}$, we set 
\[
m_k := \lfloor \memmax/(2s_k) \rfloor,
\]
where $\memmax$ is the user-defined, maximum number of basis vectors that can be allocated at once, and $s_k$ is the number of columns of the residual factors $\vC^\k, \vD^\k$ after truncation. 

The whole procedure, which combines restarting with compression, is reported in Algorithm~\ref{alg:restarted_sylv}.  Note that our pseudocode is not written to optimize memory allocation overall. For best performance, one should a priori estimate the storage needed for the final solution components $\vX_L$ and $\vX_R$ and take this into account along with $\memmax$.
\begin{algorithm}
	\caption{Restarted Sylvester solver \label{alg:restarted_sylv}}
	\begin{algorithmic}[1]
		\Procedure{Restarted\_Sylv}{$A$, $B$, $\vC$, $\vD$, $\memmax$, $k_{\max}$, $\tolres$, $\tolcomp$}
		\State Initialize  $\vX_L = [\,]$, $\vX_R = [\,]$, $\vC^\zero = \vC$, $\vD^\zero = \vD$, $\flagconv = 0$
		\For{$k = 0, \ldots, k_{\max}$}
			\State Set $m_k = \lfloor \memmax/(2s_{k})\rfloor - 2$, $s_k = \rank(\vC^\k) = \rank(\vD^\k)$
			\For{$j = 1, \ldots, m_k$}
				\State \begin{minipage}{0.8\textwidth}
				Compute (incrementally) the Arnoldi relation for $\spK_j(A,\vC^\k)$ and store $\bUU_{j+1}^\k = [\vU_1 | \cdots | \vU_{j+1}]$, $\HH_j^\k$, and $H_{j+1,j}^\k$
				\end{minipage} \label{step:krylov}
				\State \begin{minipage}{0.8\textwidth}
				Compute (incrementally) the Arnoldi relation for $\spK_j(B,\vD^\k)$ and store $\bVV_{j+1}^\k = [\vV_1 | \cdots | \vV_{j+1}]$, $\GG_j^\k$, and $G_{j+1,j}^\k$
				\end{minipage} \label{step:krylov2}
				\State Compute $Y^\k$ as the solution of the projected equation \label{step:proj-eq}
				\[
				\HH_j^\k Y + Y \big(\GG_j^\k \big)^* + \big(\bUU_j^\k \big)^* \vC^\k \big(\bVV_j^\k \vD^\k \big)^* = 0
				\]
				\State Compute residual norm $\norm{R_j^\k}$ as in \eqref{eq.res_F} or \eqref{eq.res_2}
				\If{$\norm{R_j^\k} \leq \tolres$}
					\State $m_k \leftarrow j$, $\flagconv = 1$, and go to \ref{step:SVD} \label{step:flag}
				\EndIf
			\EndFor
			\State Factor $Y^\k = \vY_L^\k \big(\vY_R^\k \big)^*$\label{step:SVD}
			\State Compute $\vZ_L^\k = \bUU_{m_k}^\k \vY_L^\k$ and $\vZ_R^\k = \bVV_{m_k}^\k \vY_R^\k$ 	\State Update $\vX_L = [\vX_L \,|\, \vZ_L^\k]$ and $\vX_R = [\vX_R \,|\, \vZ_R^\k]$
			\State $[\vX_L,\ \vX_R]\gets$ \Call{Compress}{$\vX_L$, $\vX_R$, $\tolcomp$}  \label{step:comp-sol}
			\If{$\flagconv$}
				\State \Return $\vX_L$ and $\vX_R$
			\EndIf
			\State Set $\vC^\kmore = [\vU_{m_k+1}^\k H_{m_k+1,m_k}^\k \,|\, \bUU_{m_k}^\k Y_{m_k}^\k \vE_{m_k}]$ 
			\State Set $\vD^\kmore = [\bVV_{m_k}^\k \big( Y_{m_k}^\k \big)^* \vE_{m_k} \,|\, \vV_{m_k+1}^\k G_{m_k+1,m_k}^\k]$ 
			\State $[\vC^\kmore,\ \vD^\kmore]\gets$ \Call{Compress}{$\vC^\kmore,\vD^\kmore$, $\tolcomp$} \label{step:comp-res}
		\EndFor
		\EndProcedure
	\end{algorithmic}
\end{algorithm}

\subsection{The Lyapunov equation}\label{The Lyapunov equation}
The Lyapunov equation
\begin{equation}\label{eq:lyap}
A X + X A^* + \vC \vC^* = 0,
\end{equation}
is an important special case of the more general Sylvester equation \eqref{eq:sylvester}. Indeed, in control and system theory~\cite{Antoulas2005}, it is more common to find~\eqref{eq:lyap} than the more general case. In principle, Algorithm~\ref{alg:restarted_sylv} could be directly applied to solve \eqref{eq:lyap}. However, as we will see in this section, it is possible to  enhance the algorithm by taking into account the specific structure of \eqref{eq:lyap}.

Thanks to the symmetry of \eqref{eq:lyap}, general projection techniques for Lyapunov equations generate Hermitian approximations $\Xtil = \bUU_m \Ytil \bUU_m^*$, with $\text{range}(\bUU_m) = \KK_m(A,\vC)$ and where the symmetric matrix $\Ytil$ is computed by solving the projected Lyapunov equation 
\begin{equation}\label{eq:projectedLyap}
	\HH_m Y + Y \HH_m^* + \big(\bUU_m^* \vC \big) \big(\bUU_m^* \vC \big)^* = 0, 
\end{equation}
with $\HH_m = \bUU_m^* A \bUU_m$. The norm of the residual matrix $R = A \Xtil + \Xtil A^* + \vC \vC^*$ is computed as 
\begin{equation}\label{eq:res_lyap}
	\norm{R}_F = \sqrt{2} \norm{H_{m+1,m}\vE_m^*Y_m}_F \mbox{ or } \norm{R}_2 = \norm{H_{m+1,m} \vE_m^* Y_m}_2.
\end{equation}

Employing only one approximation space is quite appealing; it potentially permits skipping the construction of the second Arnoldi relation at line~\ref{step:krylov2} of Algorithm~\ref{alg:restarted_sylv}. However, some additional considerations are needed after the first cycle because the residual matrix $R^\k$ becomes in general indefinite for $k \geq 1$. To address this, the residual is expressed in LDLT form.
Following the discussion in the previous section, at the $k$th restart the residual matrix $R^\k = A Z^\k + Z^\k A^* + R^\kless$ can be written as
\begin{align}\label{eq:symmres}
R^\k
& = [\vU_{m_k+1}^\k H_{m_k+1,m_k} \,|\, \bUU_{m_k}^\k Y^\k \vE_{m_k}]
\begin{bmatrix} 0 & I \\  I & 0 \end{bmatrix}
[\vU_{m_k+1}^\k H_{m_k+1,m_k} \,|\, \bUU_{m_k}^\k Y^\k \vE_{m_k}]^*\notag \\
& =
\vC^\kmore
\begin{bmatrix}
 0 & I \\
 I & 0 \\
\end{bmatrix}
\big(\vC^\kmore \big)^*.
\end{align}
In turn, the subspace $\KK_m(A,\vC^\kmore)$ can be used as an approximation space for the subsequent restart. The presence of the matrix $\begin{bmatrix} 0 & I \\  I & 0 \end{bmatrix}$ only affects the definition of the projected equation solved at line~\ref{step:proj-eq}, which now takes the form 
\begin{equation} \label{eq:projected_restartedLyap}
\HH_j^\k Y + Y \big(\HH_j^\k \big)^* +
\vCtil
\begin{bmatrix} 0 & I \\  I & 0 \end{bmatrix}
\vCtil^* = 0,
\qquad \vCtil = \big(\bUU_j^\k \big)^* \vC^\j.
\end{equation}
This equation can again be solved with the Bartels-Stewart method or with Hammarling's method~\cite{Hammarling1982} after splitting the right-hand side as discussed, e.g., in~\cite[Sec. 2.3]{Benner2011}. In both cases, the Hermitian structure of $\Ytil$ is preserved and can be exploited. 
 
To avoid the occurrence of complex arithmetic, an LDLT approach must be employed also during the truncation strategy, and we suggest to call Algorithm~\ref{alg:compression2} in place of Algorithm~\ref{alg:compression} at lines~\ref{step:comp-sol} and \ref{step:comp-res} of Algorithm~\ref{alg:restarted_sylv}. See, e.g., \cite{Lang2015} for similar considerations in the context of differential Riccati equations. Applying the described modifications to Algorithm~\ref{alg:restarted_sylv} returns an approximate solution of the form $X^\k = \vZ_L S \vZ_L^* \approx X$ with a (small) Hermitian matrix $S$.
\begin{algorithm}
\caption{Compression of $\vC S\vC^*$\label{alg:compression2}}
	\begin{algorithmic}[1]
		\Procedure{Compress\_Sym}{$\vC$, $S$,  $\tolcomp$}
		\State Compute Compute economy-size QR decomposition $\vC = Q_C R_C$
		\State Compute eigendecomposition $R_C S R_C^* = W \Sigma W^*$
		\State Truncate $W \Sigma W^* \approx \Wtil \Sigmatil \Wtil^*$ up to $\tolcomp$
 		\State \Return $\vCtil:= Q_C \Wtil$ and $\widetilde S := \Sigmatil$
		\EndProcedure
	\end{algorithmic}
\end{algorithm}

The final Lyapunov algorithm is stated as Algorithm~\ref{alg:restarted_lyap}.
\begin{algorithm}
	\caption{Restarted Lyapunov solver \label{alg:restarted_lyap}}
	\begin{algorithmic}[1]
		\Procedure{Restarted\_Lyap}{$A$, $\vC$,  $\memmax$, $k_{\max}$, $\tolres$, $\tolcomp$}
		\State Initialize  $\vX_L = [\,]$, $\vC^\zero = \vC$, $\flagconv = 0$, $D = I$
		\For{$k = 0, \ldots, k_{\max}$}
			\State Set $m_k = \lfloor \memmax/s_{k}\rfloor - 1$, $s_k = \rank(\vC^\k)$
			\For{$j = 1,\ldots, m_k$}
				\State \begin{minipage}{0.8\textwidth}
					Compute (incrementally) the Arnoldi relation for $\spK_j(A,\vC^\k)$ and store $\bUU_{j+1}^\k = [\vU_1 | \cdots | \vU_{j+1}]$, $\HH_j^\k$, and $H_{j+1,j}^\k$
				\end{minipage}
				\State Compute $Y^\k$ as the solution of the projected equation 
				$$
				H_j^\k Y + Y \big(H_j^\k \big)^* + \big(\bUU_j^\k \big)^* \vC^\k D \big(\bUU_j^\k \vC^\k \big)^*,
				$$
				\State Compute the residual norm $\norm{R_j^\k}$ as in \eqref{eq:res_lyap}
				\If{$\norm{R_j^\k} \leq \tolres$}
					\State $m_k \leftarrow j$, $\flagconv = 1$ and go to \ref{step:factor}
				\EndIf
			\EndFor
			\State Factor $Y^\k = \widetilde Y^\k S (\widetilde Y^\k)^*$ \label{step:factor}
			\State Update $\vX_L = [\vX_L \,|\, \bUU_{m_k}^\k \widetilde Y^\k]$
			\State $[\vX_L, S]$ $\leftarrow$ \Call{Compress\_Sym}{$\vX_L$, $\diag(I, S)$, $\tolcomp$}
			\If{$\flagconv$}
				\State \Return $\vX_L$ and $S$
			\EndIf
			\State Set $\vC^\kmore = [\vU_{m_k+1}^\k H_{m_k+1,m_k}^\k \,|\, \bUU_{m_k}^\k Y^\k \vE_{m_k}]$ 
			\State $[\vC^\kmore, D]$ $\leftarrow$ \Call{Compress\_Sym}{$\vC^\kmore$, $[0,I;I,0]$, $\tolcomp$}
		\EndFor
		\EndProcedure
	\end{algorithmic}
\end{algorithm}

\subsubsection{Positive semidefinite approximations}
A peculiar property of the Lyapunov equation~\eqref{eq:lyap} is that the solution $X$ is Hermitian positive semidefinite (SPSD) whenever $A$ is stable; in other words, all the eigenvalues of $A$ are in the open left half-plane $\C_-$ \cite{Snyders1970}. It is desirable to retain this property in an approximate solution. In the context of projection methods for Lyapunov equations, this property is ensured when $A$ is negative definite; i.e., not only $A$ but also its Hermitian part $(A+A^*)/2$ is stable. In particular, in SKSM it would follow that the matrix $\HH_m = \bUU_m^* A \bUU_m$ is stable for every $m$ and, therefore, that the solution $Y$ of the projected equation $\HH_m Y + Y \HH_m^* + (\bUU_m^* \vC) \big(\bUU_m^* \vC \big)^* = 0$, as well as the corresponding approximation $\Xtil = \bUU_m Y \bUU_m^*$, are SPSD.

In our framework, the same arguments show that $X^\zero$ is SPSD if $A$ is negative definite. However, the subsequent $Z^\k$, $0 < k \leq k_{\max}$, are in general indefinite (although still symmetric), due to the indefiniteness of the residual matrices $R^\k$. Nevertheless, it is reasonable to expect the approximate solution $X^{(\kbar)} = \sum_{k = 0}^{\bar k} Z^\k$, $0 < {\bar k} \leq k_{\max}$, to be close to a positive semidefinite matrix. In particular, if
\begin{equation}\label{eq:eigendec}
	X^{(\kbar)} = [U_+ \,|\, U_- \,|\, U_0]
	\begin{bmatrix}
	\Lambda_+ \\
		&\Lambda_-\\
		& 	& 0
	\end{bmatrix}
	[U_+ \,|\, U_- \,|\, U_0]^*
\end{equation} 
is the eigendecomposition of $X^{(\kbar)}$, partitioned according to the sign of the eigenvalues, then one can consider $X_+^{(\kbar)} := U_+ \Lambda_+ U_+^*$ as an SPSD approximation to the solution. In practice, $X_+^{(\kbar)}$ is obtained by applying a slight modification of Algorithm~\ref{alg:compression2}, which neglects the part of the eigendecomposition corresponding to the negative eigenvalues, to the matrix $X^{(\kbar)}$ returned by Algorithm~\ref{alg:restarted_sylv}. Such a step might deteriorate the accuracy of the computed solution. However, the next result shows that the error and the residual norm associated with $X_+^{(\kbar)}$ would be close to the ones associated with $X^{(\kbar)}$.
\begin{lemma}
Let the Lyapunov equation \eqref{eq:lyap} have the SPSD solution $X$. For a Hermitian approximation $X^ {(\kbar)}$, let $X_+^ {(\kbar)}$ be defined as in~\eqref{eq:eigendec} and set $R = A X^ {(\kbar)} + X^ {(\kbar)} A^* + \vC \vC^*$, $R_+ = A X_+^ {(\kbar)} + X_+^ {(\kbar)} A^* + \vC \vC^*$. Then,
\begin{align*}
	\big\|{X - X_+^{(\kbar)}}\big\| & \leq 2 \big\|{X - X^{(\kbar)}}\big\|,\\
	\norm{R_+} & \leq \norm{R} + 2 \norm{A}_2 \, \big\|{X - X^{(\kbar)}}\big\|,
\end{align*}
where $\norm{\cdot}$ corresponds to the Frobenius or the spectral norm.
\end{lemma}
\begin{proof}
Because $X^{(\kbar)}$ is Hermitian, $X_+^{(\kbar)}$ verifies
\begin{equation}\label{eq:opt}
	X_+^{(\kbar)} = {\text{argmin}}_{G \text{ is SPSD}} \norm{X^{(\kbar)} - G},
\end{equation}	
both for the Frobenius and the spectral norm \cite{Halmos,Higham}. Therefore,
\begin{align*}
	\big\|{X - X_+^{(\kbar)}}\big\|
	& \leq \big\|{X - X^{(\kbar)}}\big\| + \big\|{X^{(\kbar)} - X_+^{(\kbar)}}\big\| \leq 2 \big\|{X - X^{(\kbar)}}\big\|,
\end{align*}
where the last inequality follows from~\eqref{eq:opt} by taking into account that $X$ is SPSD.

For the second inequality, applying again \eqref{eq:opt} yields
\begin{align*}
	\big\|{A X_+^{(\kbar)} + X_+^{(\kbar)} A^* + \vC \vC^*}\big\|
	& = \big\|{R- A \big(X^{(\kbar)} - X_+^{(\kbar)} \big) + \big(X^{(\kbar)} - X_+^{(\kbar)} \big) A^*}\big\| \\
	& \leq \norm{R} + 2 \norm{A}_2 \big\|{X^{(\kbar)} - X_+^{(\kbar)}}\big\| \\
	& \leq \norm{R} + 2 \norm{A}_2 \big\|{X - X^{(\kbar)}}\big\|.
\end{align*}
\end{proof}

%% file: error_analysis.tex
\section{Residual and error analysis of Algorithm~\ref{alg:restarted_sylv}}\label{sec:error}
The compression steps performed on the intermediate residuals and solutions introduce some inexactness.  In the spirit of the analysis of inexact Krylov methods (as in, e.g., \cite{Kuerschner2018, simoncini2003theory}), we study how these compression steps affect the residual and error norms associated with the approximation returned by Algorithm~\ref{alg:restarted_sylv}. The relations retrieved in this section hold for both the Frobenius norm and the spectral norm.

Let us suppose that Algorithm~\ref{alg:restarted_sylv} terminates after $\kbar < k_{\max}$ restarts, so that $\big\|{R^{(\kbar)}}\big\| < \tolres$. Notice that the returned approximation $X^{(\kbar)}$ satisfies \[X^{(\kbar)} = \sum_{j=0}^{\kbar}(Z^\j - \Zerr^\j),\] where the matrices $\Zerr^\j$ represent the components removed by the compression step at line~\ref{step:comp-sol}. In particular, it holds that $\norm{\Zerr^\j} \leq \tolcomp$ for $j = 0, \dots, \kbar$. 

Similarly, the sequence of residuals verifies
\begin{equation}\label{eq:residuals}
	A Z^\j + Z^\j B + R^{(j-1)} - \Rerr^{(j-1)} = R^\j,
	\qquad
	j = 0, \dots, \kbar,
\end{equation}
where $\Rerr^\j$ takes into account the effect of the compression step at line~\ref{step:comp-res}; i.e., $\norm{\Rerr^\j} \leq \tolcomp$, $j = 0, \dots, \kbar$.

Summing up \eqref{eq:residuals} for $j = 0, \dots, \kbar$, we retrieve
\[
A X^{(\kbar)} + X^{(\kbar)} B + \vC \vD^*
= R^{(\kbar)} + \sum_{j=0}^{\kbar} \Rerr^\j -A \sum_{j=0}^{\kbar} \Zerr^\j - \sum_{j=0}^{\kbar} \Zerr^\j B.
\]
Therefore, the residual norm associated with $X^{(\kbar)}$ is bounded by
\begin{equation}\label{eq:res-estimate}
	\norm{AX^{(\kbar)}+X^{(\kbar)}B+\vC \vD^*}
	\leq
	\tolres + (\kbar + 1) (\norm{A} + \norm{B} + 1) \tolcomp.
\end{equation}
Equation~\eqref{eq:res-estimate} shows how the truncation tolerance $\tolcomp$ is connected with the final attainable accuracy from our restarted routine. Indeed, a reasonable way to choose $\tolcomp$ in Algorithm~\ref{alg:restarted_sylv} is such that $(k_{\max} + 1)(\norm{A} + \norm{B} + 1)\tolcomp \leq \tolres$.

We conclude by estimating  the distance from the true solution $X$. To this end, we remark that the difference $X^{(\kbar)} + \sum_{j=0}^{\kbar} \Zerr^\j - X$ exactly solves the Sylvester equation
\begin{equation}\label{eq:res-equation}
	A \Big(X^{(\kbar)} + \sum_{j=0}^{\kbar} \Zerr^\j - X \Big)
	+ \Big(X^{(\kbar)} + \sum_{j=0}^{\kbar} \Zerr^\j - X \Big) B
	= R^{(\kbar)} + \sum_{j=0}^{\kbar} \Rerr^\j.
\end{equation}
The impact of perturbing the right-hand side on the solution can be estimated via the norm of the inverse of the solution operator. This is particularly simple, when $A$, $B$ are normal and the spectra of $A$, $-B$ are separated by a vertical line in the complex plane.
\begin{lemma}
	Let $A, B \in \C^{n\times n}$ be normal matrices with  eigenvalues $\lambda_{A,j}$ and $\lambda_{B,j}$, such that
	\[
	0 \leq \Re(\lambda_{A,1})
	\leq \dots
	\leq \Re(\lambda_{A,n}),
	\qquad
	0 \leq \Re(\lambda_{B,1})
	\leq \dots
	\leq \Re(\lambda_{B,n}).
	\]
	If Algorithm~\ref{alg:restarted_sylv} terminates after $\kbar<k_{max}$ iterations then the returned solution $X^{(\kbar)}$ verifies
	\[
	\big\|{X^{(\kbar)} - X}\big\|
	\leq \frac{1}{\Re(\lambda_{A,1}) + \Re(\lambda_{B,1})}(\tolres +(\kbar+1)\tolcomp) + (\kbar +1) \tolcomp.
	\]
\end{lemma}
\begin{proof}
	By applying Theorem 1.1 in \cite{horn} to the equation \eqref{eq:res-equation} we get
	\[
	\Big\|{{X^{(\kbar)} + \sum_{j=0}^{\kbar} \Zerr^\j-X}} \Big\|
	\leq \frac{1}{\Re(\lambda_{A,1}) + \Re(\lambda_{B,1})}(\tolres +(\kbar+1) \tolcomp).
	\]
	The claim follows by using the triangular inequality and the bound \linebreak$\norm{\sum_{j=0}^{\kbar} \Zerr^\j} \leq (\kbar+1) \tolcomp$.
\end{proof}

%% file: numerical_ex_arXiv.tex
\section{Numerical experiments} \label{sec:numexperiments}

To demonstrate the efficiency of our proposed methods, we examine their behavior with respect to other viable methods on two standard problems where the coefficient matrices stem from the discretization of certain differential operators.  These other methods include the Extended Krylov Subspace Method (EKSM) (detailed, e.g., in \cite{Breiten2016, Simoncini2007} but implemented like a Rational KSM with poles at zero and infinity) and the Standard Krylov Subspace Method (SKSM) of \cite{PalittaSimoncini2018}, which is tailored to equations with symmetric coefficient matrices. 

EKSM can be implemented ``exactly" in the sense that linear systems with $A$ are solved at very high accuracy via, e.g., a direct sparse solver like the \MATLAB \emph{backslash} operator; or ``inexactly," in which inversions of $A$ are computed approximately via a block Krylov subspace method.  We test both of these variants of EKSM.  Instead of using \emph{backslash} explicitly per call to $A^\inv$, we precompute and store the Cholesky or LU factorization. In the particular case of iterative solves by $A$, we use a retooled (and possibly ILU preconditioned) block conjugate gradients (BCG) method from \cite{Dubrulle2001} when $A$ is Hermitian positive definite, and (possibly ILU preconditioned) block GMRES otherwise. We employ a rather small tolerance on the relative residual norm when BCG and block GMRES are applied to solve the linear systems with $A$. In particular, we set this tolerance to $10^{-8}$, namely two orders of magnitude smaller that the outer tolerance on the relative residual norm.  However, the novel results about inexact procedures in the basis construction of extended (and rational) Krylov subspaces presented in \cite{Kuerschner2018} may be adopted to further reduce the computational cost of the basis construction. We also remind the reader that at each EKSM iteration $2s$ new basis vectors are added to the current basis so that, at the $m$th EKSM iteration, the computed extended Krylov subspace has dimension $2(m+1)s$.
 
We would like to underscore that the comparisons between our compress-and-restart scheme, the inexact variants of EKSM, and SKSM are the fairest. Indeed, in these families of methods, only the action of $A$ on (block) vectors is allowed, making all these solvers potentially matrix-free. This is not the case for EKSM with a direct inner solver.

All methods make use of block Krylov subspace techniques and, consequently, sparse-matrix-matrix multiplication (SpMM) between $A$ and block vectors $\vV$ that could vary in size.  Since the performance of such block operations depends on machine architecture (in particular, the level 3 cache \cite{DuffMarroneRadicati1997}), memory hierarchies, and choice of libraries, we measure the performance of each method via the number calls to $A$ (``$A$-calls") and the total number of columns or column vectors to which $A$ is applied, which we refer here to as \texttt{matvec}s.\footnote{Note that, strictly speaking, a \texttt{matvec} is usually defined as the application of $A$ to a single column vector $\vv$.}  The number of $A$-calls is a crude measure of memory operations, whereas the number of \texttt{matvec}s is directly related to the amount of floating-point operations (FLOPs).  The ratio between FLOPs and memory operations is known as {\em computational intensity}, and algorithms with high computational intensity, i.e., many FLOPs to memory operations, are preferable for high-performance architectures; see, e.g., \cite{BallardCarsonDemmel2014}. We approximate computational intensity by looking at the ratio between $A$-calls and \texttt{matvec}s, which we here refer to as ``efficiency." For a more in-depth analysis of the performance of block operations and potential gains over column-by-column applications of $A$, see, e.g., the thesis by Birk \cite{Birk2015}.

Unless otherwise noted, all reported residual norms are relative and measured in the Frobenius norm.  For all experiments, the residual tolerance is set to $10^{-6}$.

All results were obtained by running \MATLAB R2017b \cite{MATLAB} on a 
standard node of the Linux cluster Mechthild hosted at the Max Planck Institute for Dynamics of Complex Technical Systems in Magdeburg, Germany.\footnote{See \url{https://www.mpi-magdeburg.mpg.de/cluster/mechthild} for further details.}

The full code to reproduce our numerical results can be found \url{https://gitlab.com/katlund/compress-and-restart-KSM}.  For the original version of SKSM, which we have adapted, see \url{https://zenodo.org/record/3252320#.XjM3UN-YVuR}.

\subsection{2D Laplacian} \label{laplacian_2d_rand_lyap}
In this example, $A$ is the second-order, centered, finite-difference discretization of the two-dimensional Laplacian operator $-(\partial_{xx} + \partial_{yy})$ on the unit square unit cube $\Omega = (0,1)^2$.  We take $n = 100$ grid points in each direction, resulting in a matrix of size $10,000 \times 10,000$.  The matrix $A$ is Hermitian positive definite.  We solve
\[
A X + X A + \vC \vC^* = 0,
\]
where $\vC \in \C^{n^2 \times 3}$ is drawn from a normal random distribution, and it is such that $\norm{\vC\vC^*}_F = 1$. In this example we call {\sc Restart\_Lyap} (Algorithm~\ref{alg:restarted_lyap}) and equipped with the enhancements described in section~\ref{The Lyapunov equation}.

Both the exact and inexact variants of EKSM need 15 iterations to meet the prescribed accuracy. As a result, a low-dimensional extended Krylov subspace of dimension 96 is constructed. We mimic such a feature by setting the memory buffer of the compress-and-restart procedure, i.e., $\memmax$, equal to $96$.  We report the results in Table~\ref{Tab.Laplacian}.
\begin{table}[htbp!]
	\begin{center}
	{\small
	\setlength\tabcolsep{0.9pt}
	\begin{tabular}{r|c|c|c|c|c|c}
        						& Its (Restarts)	& $\text{rank}(X^\k)$	& $A$-calls	& \texttt{matvec}s	& efficiency	& Time (s) \\ \hline
		{\sc Restarted\_Lyap}	& 158 (20) 			& 53  					& 158		& 1845			& 11 			& 4.02 \\
		EKSM (BCG)				& 15 (--) 			& 56					& 2615		& 7845			& 3 			& 8.29 \\
		EKSM (BCG+ILU) 			& 15 (--)			& 56					& 900 		& 2700 			& 3				& 4.17 \\
		EKSM (exact)			& 15 (--) 			& 56  					& 30		& 90			& 3 			& 0.34 \\
		SKSM (two-pass Lanczos)	& 148 (--) 			& 65					& 295		& 885			& 3				& 4.08 \\
	\end{tabular}
	}
	\caption{Example~\ref{laplacian_2d_rand_lyap}. Performance measures. $s = 3$, $\memmaxmath=96$. \label{Tab.Laplacian}}
	\end{center}
\end{table}

Our routine {\sc Restarted\_Lyap} needs 20 restarts to converge for a total number of 158 iterations. The compress-and-restart scheme lets us maintain a low storage demand and, at each restart, a polynomial Krylov subspace of dimension (at most) 96 is constructed, as in the case of the comparable EKSM. In SKSM with two-pass Lanczos, thanks to the symmetry of $A$, we can use short-term recurrences so that the whole basis is never stored, only three basis vectors at a time. The low-rank factors of the solution are recovered by means of a two-pass strategy; see \cite{PalittaSimoncini2018} for more details. Despite the very low storage demands of SKSM, the number of $A$-calls exceeds that of the compress-and-restart scheme.

Both the numbers of $A$-calls and \texttt{matvec}s of {\sc Restarted\_Lyap} are much lower than those accrued by the inexact procedures EKSM (BCG) and EKSM (BCG+ILU). Moreover, our procedure is more efficient for block operations while maintaining a computational time comparable to that of the other SpMM-dominant routines, thus reinforcing its potential for further speed-ups in communication-dominant high-performance environments. Timings for EKSM (exact) largely benefit from having precomputed and stored the Cholesky factors of $A$.

All the routines we tested return a low-rank numerical solution and, for this example, the approximation computed by {\sc Restarted\_Lyap} has the lowest rank. In Figure~\ref{fig:lapl2D_rpk} (left) we report the ranks of both the residual and the approximate solution computed at the end of each restart, together with the rank of the final solution.  These results illustrate that the truncation strategy {\sc Compress\_Sym} (Algorithm~\ref{alg:compression2}) is able to maintain a moderate rank in the residual $R^\k$, for all $k = 0, \ldots, 20$. This is crucial for making the construction of the subsequent restart space feasible, when necessary.
\begin{figure}\label{fig:lapl2D_rpk}
	\begin{center}
	\begin{tabular}{cc}
		{\begin{tikzpicture}
			\begin{axis}[width=.48\linewidth, height=.32\textheight, legend pos = north east, xlabel = Cycle index ($k$), ymax = 90]
				\addplot table[x index=0, y index=1] {./rpk_n10000_r3_m96_res_ranks.dat};
				\addplot table[x index=0, y index=1] {./rpk_n10000_r3_m96_sol_ranks.dat};
				\legend{\tiny{$\rank\big(R^\k\big)$}, \tiny{$\rank\big(X^\k\big)$}};
			\end{axis}
		\end{tikzpicture}}
		&
		{\begin{tikzpicture}
			\begin{semilogyaxis}[width=.48\linewidth, height=.32\textheight, legend pos = north east, xlabel = Iteration index, ymax= 1]
				\addplot[violet] table[x index=0, y index=1] {./rpk_n10000_r3_m96_residual_norms.dat};
				\addplot[only marks, mark = |, gray, mark options={scale=6}] table[x index=0, y index=1] {./rpk_n10000_r3_m96_cycle_markers.dat};
				\legend{\tiny{residual norm}, \tiny{restart}};
			\end{semilogyaxis}
		\end{tikzpicture}}
	\end{tabular}
	\end{center}
	\caption{Example~\ref{laplacian_2d_rand_lyap}. Residual and solution ranks (left) and residual norms (right) for the compress-and-restart polynomial Krylov method. Vertical tick marks indicate the start of a new cycle.}
\end{figure}
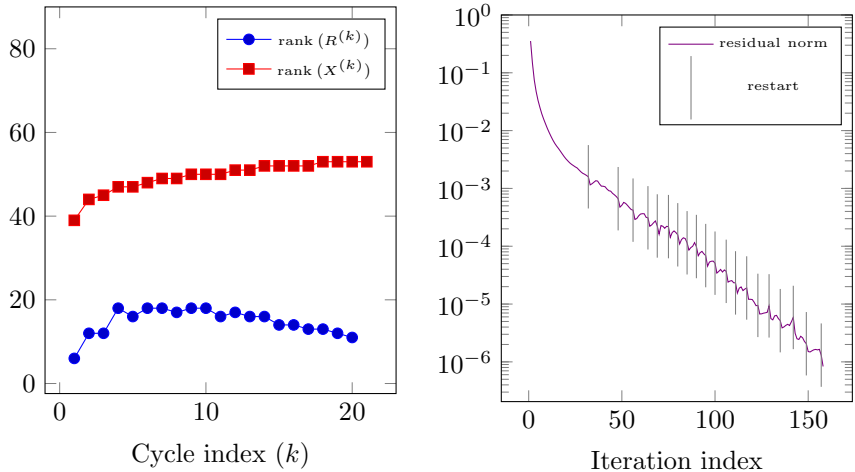
In Figure~\ref{fig:lapl2D_rpk} (right) we also plot, in logarithmic scale, the relative residual norm history through all the 158 iterations performed by {\sc Restarted\_Lyap}. We can see that the relative residual norm does not have a smooth behavior. This is due to the a Galerkin condition we impose on the residual. Even though this phenomenon has not been extensively analyzed in the matrix equation literature yet, it is quite well understood in the linear system setting. See, e.g., \cite{Cullum1996,Cullum1995}. Imposing a minimal residual condition in place of a Galerkin condition might be beneficial, although such a strategy has some peculiar shortcomings in the matrix equation framework. See, e.g., \cite{Lin2013, HuReichel1992, Palitta2019}.

We now illustrate some observations about the possible computation of an indefinite approximate solution. See also the end of section~\ref{The Lyapunov equation}. In Figure~\ref{fig:lapl2D_eig} (left) the blue circles denote the minimum nonzero eigenvalue of $X^\k$ for $k=0,\ldots,20$, and of the final solution. The black dashed line marks the $x$-axis. We can appreciate how these eigenvalues are all positive so that $X^\k$ is positive semidefinite for all $k$.
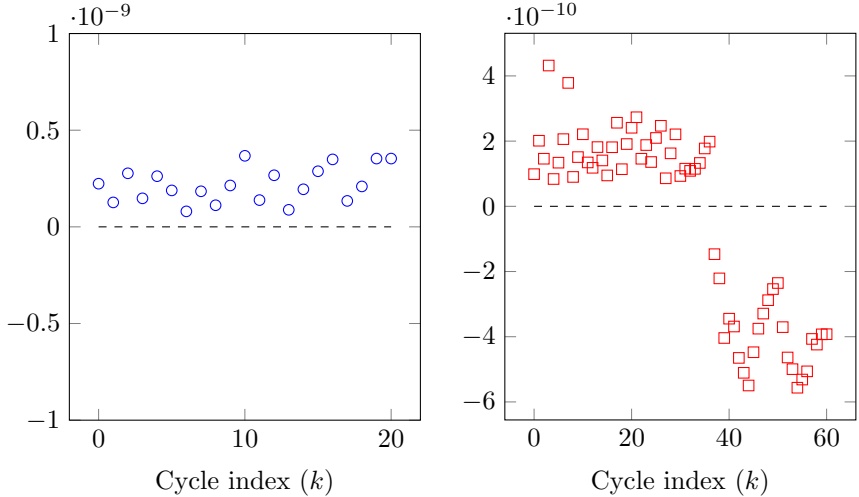
\begin{figure}\label{fig:lapl2D_eig}
	\begin{center}
	\begin{tabular}{cc}
	{\begin{tikzpicture}
		\begin{axis}[width=.48\linewidth, height=.32\textheight, legend pos = north east, xlabel = Cycle index ($k$), ymax= 1e-9, ymin= -1e-9]
			\addplot[mark=o, color=blue, only marks] table[x index=0, y index=1] {./rpk_n10000_r3_m96_eig.dat};
			\addplot[dashed, color=black] table[x index=0, y index=2] {./rpk_n10000_r3_m96_eig.dat};
		\end{axis}
	\end{tikzpicture}}
		&
		{\begin{tikzpicture}
		\begin{axis}[width=.48\linewidth, height=.32\textheight, legend pos = north east, xlabel = Cycle index ($k$)]
			\addplot[mark=square, color=red, only marks] table[x index=0, y index=1] {./rpk_n10000_r3_m96_eig_nonnorm.dat};
			\addplot[dashed, color=black] table[x index=0, y index=2] {./rpk_n10000_r3_m96_eig_nonnorm.dat};
		\end{axis}
	\end{tikzpicture}}
	\end{tabular}
	\end{center}
	\caption{Example~\ref{laplacian_2d_rand_lyap}. Smallest nonzero eigenvalue of $X^\k$ in the case of a normalized right-hand side (left) and non-normalized right-hand side (right).}
\end{figure}

We now consider the same Lyapunov equation as before but we do not normalize the random matrix $\vC \vC^*$.  This is now a harder problem for the polynomial Krylov subspace method and, with $\memmax = 96$, {\sc Restarted\_Lyap} needs 60 restarts to converge.
The large number of restarts is due to the large rank of $R^\k$, especially for large $k$, which limits us to a couple of iterations per restart in order to stay within the memory buffer prescribed by $\memmax$. In Figure~\ref{fig:lapl2D_eig} (right) we report the minimum nonzero eigenvalue of $X^\k$ for $k = 0, \ldots, 60$. The matrix $X^\k$ stops being positive semidefinite for $k \geq 39$, and we thus apply the strategy presented at the end of section~\ref{The Lyapunov equation} to compute $X^\k_+$, $k=60$, in place of $X^\k$.
For this example, such a strategy has multiple advantages. Indeed, in addition to providing a positive semidefinite approximate solution, it reduces the rank of the computed solution while maintaining the prescribed accuracy. In particular, $\rank(X^\k) = 81$ while $\rank\big(X_+^\k \big) = 77$ and $\norm{A X_+^\k + X_+^\k A^* + \vC \vC^*}_F / \norm{\vC \vC^*}_F = 8.84 \times 10^{-7}$.

We conclude this example by showing one of the most remarkable features of our novel compress-and-restart strategy. The total memory demand of a Krylov subspace method is in general difficult to predict a priori, as it requires knowing the dimension of the subspace in which a satisfactory approximation can be found.  If a situation calls for stringent memory management, then methods that increase the basis size every iteration, like EKSM, may not be able to reach the desired accuracy before exhausting memory resources.  Table~\ref{Tab.Laplacian2} demonstrates the superiority of {\sc Restarted\_Lyap} in precisely such a scenario, where $\memmax = 250$ and the right-hand side $\vC \vC^*$, $\vC \in \R^{n^2\times s}$, $s = 25$, is a low-rank approximation of the matrix $C \in \R^{n^2\times n^2}$ representing the discretization of $\exp((x_1^p + x_2^p + x_3^p + x_4^p)^{1/p})$, $p = 2$, on the hypercube $[-1,1]^4$.
{\renewcommand{\arraystretch}{1.2}
\begin{table}
	\begin{center}
	{\small
	\begin{tabular}{r|c|c}
        						& Its (Restarts)	& Rel. Res. \\ \hline
		{\sc Restarted\_Lyap}	& 165 (33) 			& $7.06\times 10^{-7}$\\
		EKSM (BCG)				& 5 (--) 			&  $1.51\times 10^{-4}$ \\
		EKSM (BCG+ILU) 			& 5 (--)			& $1.60\times 10^{-4}$ \\
		EKSM (exact)			& 5 (--) 			& $1.43\times 10^{-4}$ \\
		SKSM (two-pass Lanczos)	& 69 (--) 			& $5.76\times 10^{-4}$ \\
	\end{tabular}
	}
	\vspace{0.1cm}
	
	\caption{Example~\ref{laplacian_2d_rand_lyap}. Performance measures. $s=25$, $\memmaxmath=250$.}\label{Tab.Laplacian2}
	\end{center}
\end{table}
}
All the EKSM variants are forced to stop as soon as a space of dimension 250 is constructed; in tests not reported here, we found that $\memmax = 400$ allows EKSM to reach the desired residual tolerance. SKSM with two-pass Lanczos seems to suffer from the large rank of $\vC$. Indeed, the computed residual norm differs from the actual one by some orders of magnitude, likely due to the loss of orthogonality in the computed basis. A full, or perhaps even partial, re-orthogonalization of the basis may fix this issue but doing so in a memory-sensitive manner remains open.  On the other hand, {\sc Restarted\_Lyap} successfully reaches the desired residual tolerance, thus demonstrating its potential in not only memory-limited situations but also for matrix equations whose right-hand side has high rank.

\subsection{Convection-diffusion equation} \label{conv_diff_3d_sylv}
We turn our attention to the main problem, a Sylvester equation of the form \eqref{eq:sylvester}, where the coefficient matrices $A$ and $B$ stem from the second-order, centered, finite-difference discretization of the 3D convection-diffusion operators
\begin{equation}\label{Ex.2:operators}
\mathcal{L}_A(u) =
-\varepsilon \Delta u + \vec{w}_A \cdot \nabla u,
\quad
\mathcal{L}_B(u) =
-\varepsilon \Delta u + \vec{w}_B \cdot \nabla u,
\end{equation}
on the unit cube $\Omega = (0,1)^3$, respectively. The viscosity parameter is $\varepsilon = 0.01$ while the convection vectors $\vec{w}_A$, $\vec{w}_B$ are defined as 
\[
\vec{w}_A = (x \sin(x), y \cos(y), e^{z^2-1}),
\quad
\vec{w}_B = (yz(1-x^2), 0, e^z).
\]
If the operators in~\eqref{Ex.2:operators} are discretized  with $n$ equidistant nodes in each direction, then the nonsymmetric matrices $A$ and $B$ are each of dimension $n^3$. We consider two different problem sizes for \eqref{eq:sylvester}, $n = 25$ and $n = 80$.  The resulting problems are of dimension $15625$ and $512000$, respectively.  The low-rank matrices $\vC, \vD \in \C^{n^3 \times 3}$ are once again drawn from a normal random distribution, and they are such that $\norm{\vC \vD^*}_F = 1$.  
 
For $n = 25$, all the EKSM variants we tested-- EKSM (exact), EKSM (BGMRES), and EKSM (BGMRES+ILU)-- need 21 iterations to converge, in which case, two extended Krylov subspaces of dimension 132 are constructed. As before, we set the memory buffer of our compress-and-restart routine equal to the memory consumption of EKSM; i.e., we set $\memmax$ in Algorithm~\ref{alg:restarted_sylv} equal to 264. With this setting, {\sc Restarted\_Sylv} needs 2 restarts for a total of 85 iterations to converge, and it computes an approximate solution of rank 57, equal to the rank of the solution returned by all the EKSM variants.
 
In Figure~\ref{Ex.2:figure} (left) we depict the ranks of both the residual and of the approximate solution computed at each {\sc Restarted\_Sylv} restart, together with the rank of the final solution, while in Figure~\ref{Ex.2:figure} (right), the entire relative residual norm history is reported.  We again see that the low-rank compression procedure (Algorithm~\ref{alg:compression}) is able to maintain a moderate rank in both the residual and the approximate solution.
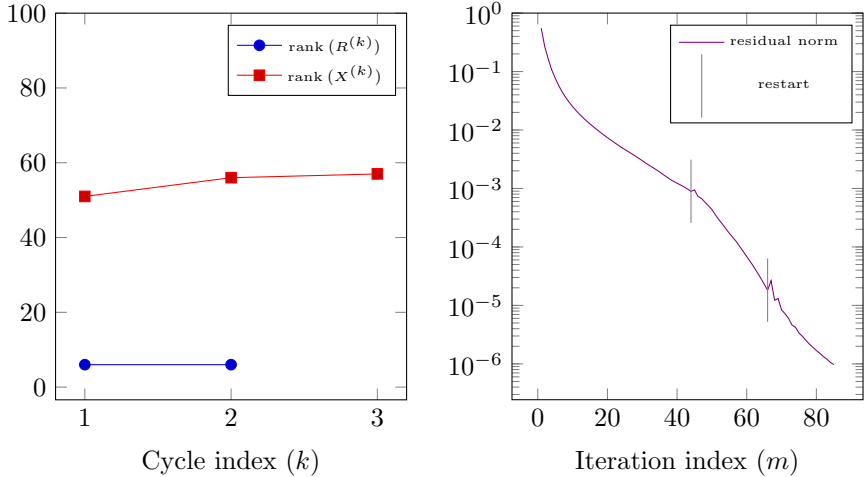
\begin{figure}\label{fig:conv_diff_rpk}
	\begin{center}
	\begin{tabular}{cc}
		{\begin{tikzpicture}
			\begin{axis}[width=.48\linewidth, height=.32\textheight, legend pos = north east, xlabel = Cycle index ($k$), ymax = 100]
				\addplot table[x index=0, y index=1] {./rpk_n15625_r3_m264_res_ranks.dat};
				\addplot table[x index=0, y index=1] {./rpk_n15625_r3_m264_sol_ranks.dat};
				\legend{\tiny{$\rank\big(R^\k\big)$}, \tiny{$\rank\big(X^\k\big)$}};
			\end{axis}
		\end{tikzpicture}}
		&
		{\begin{tikzpicture}
			\begin{semilogyaxis}[width=.48\linewidth, height=.32\textheight, legend pos = north east, xlabel = Iteration index ($m$), ymax= 1]
				\addplot[violet] table[x index=0, y index=1] {./rpk_n15625_r3_m264_residual_norms.dat};
				\addplot[only marks, mark = |, gray, mark options={scale=6}] table[x index=0, y index=1] {./rpk_n15625_r3_m264_cycle_markers.dat};
				\legend{\tiny{residual norm}, \tiny{restart}};
			\end{semilogyaxis}
		\end{tikzpicture}}
	\end{tabular}
	\end{center}
	\caption{Example~\ref{conv_diff_3d_sylv}, $n = 25$. Residual and solution ranks (left) and residual norms (right) for the compress-and-restart polynomial Krylov method. Vertical tick marks indicate the start of a new cycle.}\label{Ex.2:figure}
\end{figure}

We remind the reader that two different subspaces have to be generated, one by $A$ and the other by $B$, when Sylvester equations are solved by projection techniques. The computational efforts devoted to such tasks may significantly differ from each other if the matrices $A$ and $B$ have dissimilar spectral properties. This can be appreciated by looking at the results in Table~\ref{table:conv_diff}. In this example, iteratively solving linear systems with $B$ requires more iterations than with $A$. Therefore, a larger number of $B$-calls is required in EKSM (BGMRES), even though the extended Krylov subspaces generated by $A$ and $B$ have the same dimensions.  This phenomenon is reversed in EKSM (BGMRES+ILU), indicating that the ILU preconditioner for $B$ performs better than the one for $A$.  Our compress-and-restart procedure is not influenced by such issues, though, since by design, the spaces for $A$ and $B$ are computed to the same dimension (assuming no breakdowns).  It is indeed possible, and in some cases desirable, to allow for different basis sizes for $A$ and $B$, especially if the operators differ significantly in size or complexity.  However, this flexibility is not trivial to implement, especially with the compression at step~\ref{step:comp-res} of Algorithm~\ref{alg:restarted_sylv}, which could cause $\vC^\kmore$ and $\vD^\kmore$ to have different numbers of columns.  For the simplicity of presentation and implementation, we therefore do not explore this option further in the present work.

\begin{table}
	\begin{center}
		{\small
		\setlength\tabcolsep{1.3pt}
			\begin{tabular}{r|c|c|c|c|c|c}
				& $A$-calls	& \texttt{matvec}s	& efficiency	& B-calls	& \texttt{matvec}s	& efficiency \\ \hline
				{\sc Restarted\_Sylv}	& 85	& 378	& 4	& 85	& 378	& 4 \\
				EKSM (BGMRES)  			& 1819	& 5457	& 3 	& 2445	& 7335	& 3 \\
				EKSM (BGMRES+ILU)  		& 455	& 1365	& 3 	& 376	& 1128	& 3 \\
				EKSM (exact)		& 42 	& 126  	& 3		& 42	& 126 	& 3 \\
			\end{tabular}
		}
		\caption{Example~\ref{conv_diff_3d_sylv}. Performance measures. $n=25$, $s=3$, $\memmaxmath = 264$. \label{table:conv_diff}}
	\end{center}
\end{table}

The extra memory allocation required by the iterative solution of linear systems with $A$ and $B$ during the basis construction must be taken into account for precisely identifying the memory demands of EKSM (BGMRES) and EKSM (BGMRES+ILU). To the best of our knowledge, such an issue has not yet been rigorously explored in the literature, and it should not be underestimated. Increasing the density of the discretization grid to $n = 80$ (for a problem size of $512000$), leads to 26 iterations for EKSM (exact) to converge and the construction of two extended Krylov subspaces of dimension 162 each. With $\memmax = 324$, EKSM (BGMRES) and EKSM (BGMRES+ILU) are not able to achieve the prescribed accuracy, because at some point, the number of (outer) extended Krylov basis vectors already computed plus the number of vectors needed by the (inner) block polynomial Krylov subspace to accurately compute the next (outer) basis vector exceeds $\memmax$.

We conclude this example by pointing out that {\sc Restarted\_Sylv} turns out to be competitive also in terms of computational time for both $n = 25$ and $n = 80$; see Table~\ref{table2:conv_diff}.  Indeed, a preconditioner tailored to the problem may benefit EKSM with inexact solves, but the design of an effective preconditioner is a difficult task and largely problem-dependent.  Our compress-and-restart procedure achieves excellent performance without the need for preconditioning while still managing severe memory limitations.
\begin{table}
	\begin{center}
		{\small
			\begin{tabular}{r|c|c}
				& \multicolumn{2}{c}{Time (s)} \\ 
				& $n=25$, $\memmax = 264$ &
				$n=80$, $\memmax = 324$ \\ \hline
				{\sc Restarted\_Sylv}	& 5.96 		& 549.99 \\
				EKSM (BGMRES)			& 298.15 	& -- \\
				EKSM (BGMRES+ILU)		& 14.62 	& -- \\
				EKSM (exact)			& 3.44 		& 589.72\\
			\end{tabular}
		}
		\caption{Example~\ref{conv_diff_3d_sylv}. Computational times for different values of $n$. \label{table2:conv_diff}}
	\end{center}
\end{table}

%% file: conclusions.tex
\section{Conclusions}
Modern computing architectures pose many challenges demanding the optimization of not only operation counts (FLOPs) but also memory allocation and movement.  Much work has been devoted to adapting iterative methods for large and sparse linear systems to these architectures, but straightforward extensions of successful strategies for linear systems to matrix equations are not always feasible. We have demonstrated how to apply a common and effective Krylov subspace technique for linear systems, namely restarts, by introducing a compression step to mitigate the growing rank of the residual.  The resulting compress-and-restart method is viable for both Sylvester and Lyapunov matrix equations, and given a fixed memory requirement, it tunes the computable basis size automatically at each restart.  Compared to extended Krylov subspace methods (EKSM), which require an inner solver for applications of $A^\inv$ and therefore either well-designed preconditioners or fast sparse Cholesky and LU factorizations, our compress-and-restart polynomial Krylov methods require very little set-up and converge with competitive timings in situations where (unrestarted) EKSM run out of memory. 

Our compress-and-restart method is a success not only for matrix equations but potentially other classes of higher-order problems where memory resources are even more limited, due to the curse of dimensionality.
Moreover, our compress-and-restart paradigm is not restricted to the use of block polynomial Krylov subspaces; different approximation spaces can be used as well.

\section*{Acknowledgments}
 The third and the fourth authors are members of the Italian INdAM Research group GNCS.